\documentclass{amsart}
\usepackage{amssymb,amsmath,latexsym}
\usepackage{amsthm}
\usepackage{fontenc}
\usepackage{amssymb}
\numberwithin{equation}{section}
\newtheorem{theorem}{Theorem}[section]
\newtheorem{corollary}{Corollary}[theorem]
\newtheorem{lemma}[theorem]{Lemma}

\begin{document}
\author{Alexander Eric Patkowski}
\title{A Note on a Tauberian theorem for arithmetic functions}

\maketitle
\begin{abstract}We offer new Tauberian theorems for a generalized partition function as our main result. Our analysis provides insight into asymptotic behavior of power series with arithmetic functions as coefficients.\end{abstract}



\section{Introduction}
Tauberian theorems have a rich history in classical analysis providing asymptotic behavior of power series with conditions on its coefficients. The well-known Hardy-Littlewood Tauberian result [9, pg.157] states that if $a_n$ are non-negative, constant $C>0,$ and 
$$\sum_{n\ge0}a_nz^n\sim C\frac{1}{1-z},$$ as $z\rightarrow1^{-}$ in $\mathbb{R},$
then $$A(N):=\sum_{0\le n\le N}a_n\sim CN,$$ as $N\rightarrow\infty.$
Here we use the usual definition in the sense that $f(x)=O(g(x))$ means there exists a constant $C>0$ such that $|f(x)|\le Cg(x).$ Additionally we write $f(x)\sim g(x),$ to mean that $\lim_{x\rightarrow x_0}f(x)/g(x)=1.$ The main class of asymptotic formula we will present in this note follow directly from the following well-known result [5] (also found in [4, Lemma 1]).
\begin{lemma} Suppose that the sequence $(a_n)$ is real, and $\sum_{0\le j\le n}a_j=O(n^{\epsilon}f(n)),$ $\epsilon>0,$ as $n\rightarrow\infty.$ Here the limit is taken to be positive or infinite. Then we have that 
 \begin{equation}\sum_{n\ge0}a_nz^n\sim C\frac{\Gamma(\epsilon+1)}{(1-z)^{\epsilon+1}}f(\frac{1}{1-z}).\end{equation}
 as $z\rightarrow1,$ where $\Gamma(\epsilon)$ is the classical Gamma function.
\end{lemma}
Arithmetic sums of the form $\sum_{n\le x}a_nb_n,$ where $b_n$ is chosen as an arithmetic function $b_n:\mathbb{N}\rightarrow\mathbb{C},$ have been studied for some time. In the case of $b_n=\Lambda(n),$ the von Mangoldt function, estimates were proved in [3]. A recent study on sums of the form $\sum_{n\le x}a_nd_r(n),$ where $d_r(n)$ denotes the $r$-fold divisor function, was produced in [8]. \par The purpose of this note is to offer asymptotic results for power series of the form $$\sum_{n\ge0}a_nb_nz^n,$$ as $z\rightarrow1^{-}.$ Our main results are centered on partitions in the next section, and some additional results follow in the last section concerning the case $b_n=\Lambda_{k}(n)=\sum_{d|n}\mu(d)\left(\log(\frac{n}{d})\right)^k,$ where $\mu(n)$ is the M$\ddot{o}$bius function [6,9].

\section{The partition function $p_H(n)$}
First let us consider $p_m(n),$ the number of partitions of $n$ into at most $m$ parts. It is an elementary fact [1, pg.213] that $p_m(n)\le (n+1)^m.$ Since $(x+1)^m=\sum_{i}\binom{m}{i}x^i,$ we have the trivial estimate 
$$\sum_{n\le x}a_np_m(n)=O(x^mA(x)),$$ which suggests an interesting application of Lemma 1.1 would be of interest in understanding $$\sum_{n\ge0}a_np_m(n)z^n.$$ However, it seems we can do better by appealing to a result found in [7].
\begin{theorem} Let $H$ denote the set consisting of $k$ positive integers which have a greatest common divisor of $1$ (i.e. are relatively prime). If $p_H(n)$ is the number of partitions of $n$ into parts taken from the set $H,$ we have that
\begin{equation}\sum_{n\ge0}a_np_{H}(n)z^n=O\left(\frac{1}{(1-z)^{k-1}}A(\frac{1}{1-z})\right),\end{equation} as $z\rightarrow1^{-},$ where the implied constant depends on the set $H.$
\end{theorem}
\begin{proof} First we need [7]
\begin{equation} p_H(n)=\left(\frac{1}{\prod_{h\in H}h}\right)\frac{n^{k-1}}{\Gamma(k)}+O(n^{k-2}).\end{equation} We may temporarily view the right hand side of (2.2) as a polynomial of degree $k-1$ with indeterminate $n$ and leading coefficient depending on $k$ and the set $H.$ Hence we may use standard properties of $O$ to write
\begin{equation} p_H(n)=O(n^{k-1}),\end{equation} where the implied constant depends on the set $H.$ Consequently, we find that 
\begin{equation} A_{H}(x):=\sum_{n\le x}a_np_H(n)=O(x^{k-1}A(x)).\end{equation} Using summation by parts 
\begin{equation}\sum_{n\ge0}a_np_{H}(n)z^n=(1-z)\sum_{n\ge0}A_{H}(n)z^n,\end{equation}
Applying (2.4) and (2.5) to Lemma 1.1 with $\epsilon=k-1$ now gives the theorem.
\end{proof}
Next we consider a direct corollary of this result by applying to prime number theorem [6, pg.31].
\begin{corollary} We have that
\begin{equation}\sum_{p>1}p_{H}(p)z^p=O\left(\frac{1}{(1-z)^{k}\log(\frac{1}{1-z})}\right),\end{equation} 
as $z\rightarrow1^{-},$ where the sum over $p$ on the left side of (2.6) is over primes.
\end{corollary}
\begin{proof} The Prime Number Theorem [6, pg.31] states that 
\begin{equation}\sum_{p\le x}1\sim \frac{x}{\log(x)}.\end{equation} It is easy to see that the right side of (2.7) satisfies the growth condition in Lemma 1.1 by applying L'Hospital's rule. Choosing $a_n$ in Theorem 2.1 to be $1$ if $n=p$ is a prime, and $0$ otherwise, we see that the Corollary follows from applying (2.7).
\end{proof}
Our last result of this section is a general estimate for the $a_n=1$ case of the sum $A_{H}(x).$ Recall that the Bernoulli polynomials are generated by [2, Definition 9.1.1],
$$\frac{te^{tx}}{e^t-1}=\sum_{n\ge0}\frac{B_{n}(x)}{n!}t^n.$$ 
\begin{theorem} Let $B_k(x)$ denote the $k$th Bernoulli polynomial. For positive $x$ and $k\ge2,$
\begin{equation}\sum_{n\le x}p_{H}(n)=\left(\frac{1}{\prod_{h\in H}h}\right)\frac{1}{\Gamma(k)}\left(\frac{B_k(x+1)-B_k(0)}{k}\right)+O(x^{k-2}).\end{equation}
\end{theorem}
\begin{proof} We need the formula [2, pg.31, Proposition 9.2.12]
\begin{equation}\sum_{1\le m\le x}m^{k}=\frac{B_{k+1}(x+1)-B_{k+1}(0)}{k+1},\end{equation} valid for each $k\ge1.$ Summing (2.2) over the interval $1\le n\le x$ and applying (2.9) gives the theorem after using standard properties of $O$ for polynomials.
\end{proof}
A nice corollary to this result is
$$\sum_{n\le x}p_{H}(n)\sim\left(\frac{1}{\prod_{h\in H}h}\right)\frac{1}{\Gamma(k)}\left(\frac{B_k(x+1)-B_k(0)}{k}\right),$$ as $x\rightarrow\infty.$
\section{Remarks related to the von Mangoldt function} 
As was previously noted, the sum $\sum_{n\le x}a_n\Lambda(n)$ was examined in [3]. It was shown [6, pg.339] \it heuristically \rm that
\begin{equation} \sum_{n\le x}a_n\Lambda(n)\sim S A(x),\end{equation} as $x\rightarrow\infty,$ where $S=-\sum_{n\le M}\mu(n)\log(n)g(n),$ and $g$ is a suitable multiplicative function such that $S$ is finite when $M\rightarrow\infty.$ If (3.1) holds true, and $A(x)\sim x^{\epsilon},$ $\epsilon>0,$ it would be the case then that we have
$$\sum_{n\ge1}a_n\Lambda(n)z^n\sim C\frac{1}{(1-z)^{\epsilon}}.$$ A simple special case of this would be if $a_n=1$ for all $n$, which would imply 
$$\sum_{n\ge1}\Lambda(n)z^n\sim C\frac{1}{(1-z)}.$$ This we know to be true by the Prime Number Theorem $\sum_{n\le x}\Lambda(n)\sim x$ [6, pg.31, eq.(2.5)], and the Hardy-Littlewood result noted in the introduction [9, pg.157]. \par We give a general related result which we believe to be of some interest.
\begin{theorem} Suppose that $A(x)=O(x^{\epsilon}),$ $\epsilon>0,$ as $x\rightarrow\infty.$ We have that 
\begin{equation}\sum_{n\ge1}a_n\Lambda_k(n)z^n=O\left(\frac{1}{(1-z)^{\epsilon}}\left(\log(\frac{1}{1-z})\right)^{k}\right),\end{equation} as $z\rightarrow1^{-}.$
\end{theorem}
\begin{proof} Using the fact that $\Lambda_k(n)\le (\log(n))^k$ [6, pg.16, eq.(1.45)], we have that 
$$\sum_{n\le x}a_n\Lambda_k(n)\le (\log(x))^kA(x).$$
This together with the growth assumption on $A(x)$ and Lemma 1.1 gives the theorem.
\end{proof}
A nice simple consequence of applying the Prime Number Theorem $\sum_{n\le x}\Lambda(n)\sim x,$ to Theorem 3.1 gives us the formula
\begin{equation}\sum_{n\ge1}\Lambda(n)\Lambda_k(n)z^n=O\left(\frac{1}{(1-z)}\left(\log(\frac{1}{1-z})\right)^{k}\right),\end{equation} as $z\rightarrow1^{-}.$
Setting $k=1$ in (3.3) gives us the elegant
$$\sum_{n\ge1}\Lambda^2(n)z^n=O\left(\frac{1}{(1-z)}\log(\frac{1}{1-z})\right),$$ as $z\rightarrow1^{-}.$

1390 Bumps River Rd. \\*
Centerville, MA
02632 \\*
USA \\*
E-mail: alexpatk@hotmail.com, alexepatkowski@gmail.com
\end{document}